\documentclass[preprint,11pt]{elsarticle}
\usepackage[utf8]{inputenc}

\def\PP{{\mathbb P}}

\usepackage{natbib} 
\usepackage{amsmath,amssymb,epsfig,epstopdf,graphicx,cases,
subfigure,float, url, color, bbm, lscape, bm,enumitem,caption,array}
\usepackage{amsthm}
\usepackage{booktabs,multirow,makecell}
\newtheorem{theorem}{Theorem}[section]

\newtheorem{definition}{Definition}[section]
\newtheorem{lemma}{Lemma}[section]

\newtheorem{assumption}{Assumption}
\usepackage{multirow}
\usepackage{mathrsfs}
\usepackage[linesnumbered,ruled]{algorithm2e}
\usepackage{booktabs}

\usepackage{geometry}
\graphicspath{{./Figures/}}
\begin{document}
\begin{frontmatter}

\title{Maintenance optimization of a two-component system with mixed observability}
\author[add1,add2]{Nan Zhang\corref{cor1}}
\ead{nan.zhang@bit.edu.cn }
\cortext[cor1]
{Corresponding author}
\address[add1]{School of Management, Beijing Institute of Technology,  Beijing, China}
\address[add2]{State Key Laboratory of Explosion Science and Safety Protection, Beijing Institute of Technology, Beijing, China}

\author[add3]{Inmaculada T. Castro}
\ead{inmatorres@unex.es}
\address[add3]{Department of Mathematics, University of Extremadura, Caceres, Spain}

\author[add4]{Maria Luz Gamiz}
\ead{mgamiz@ugr.es}
\address[add4]{Department of Statistics and Operational Research, University of Granada, Spain}
\begin{abstract}
This paper studies maintenance optimization for a two-component system under mixed observability. Component~$U_1$ is fully monitored, whereas component~$U_2$ is only partially observable due to sensing limitations. The system exhibits unidirectional positive degradation dependence, in which the health state of component~$U_1$ influences the degradation process of component~$U_2$, but not vice versa. We propose a novel framework for modeling and optimizing maintenance decisions for such systems using a partially observable Markov decision process (POMDP). Under mild conditions, we analytically establish structural properties of the optimal maintenance policy. Baum-Welch algorithm with multiple sample paths is developed to estimate the unknown 
system parameters in the context of a covariate-dependent Hidden Markov Model.
Numerical experiments demonstrate the effectiveness of the proposed parameter estimation and the maintenance policy. Across 64 instances, we  show that it consistently outperforms classical threshold-based policies. Specifically, when the degradation of component $U_1$ is faster, it achieves maximal cost reductions of up to approximately $6\%$.
\end{abstract}
\begin{keyword}
Maintenance;
Mixed observability;
Baum-Welch algorithm; 
Hidden Markov model; 
Partially observable Markov
decision process
\end{keyword}
\end{frontmatter}

\section{Introduction}\label{sec_intro}
Modern engineering systems—such as industrial machinery, automotive systems, and wind turbines—often consist of heterogeneous components with differing sensing and monitoring capabilities. In many of these systems, {mixed observability} is a common phenomenon: some components are fully observable, with their health states perfectly revealed, while others remain only partially observable due to cost constraints or intrinsic inaccessibility.

In  such systems, the state of the fully observable components often directly influences the degradation process of the partially observable ones, while the former remains independent of the latter, giving rise to what we term {
unidirectional positive degradation dependence (UPDD)}. For instance, in manufacturing plants, a motor’s state may be continuously monitored using high-precision sensors, whereas the state of a gear such as its surface wear is partially reflected through indicators like vibration patterns or noise levels. Here, UPDD manifests in that when the motor operates under overheating or high-load conditions, it imposes excessive mechanical stress on the gear by accelerating its wear \cite{mao2023transmission}. In a wind turbine system, the generator’s state is generally directly monitored through sensors tracking output power, temperature, and vibration, $etc.,$ whereas blade conditions
  such as micro-cracks or surface erosion are only partially revealed through indicators such as reduced power efficiency or abnormal vibrations \cite{rezamand2020critical}. Operating the generator in high wind speed will increase mechanical stress on the blades, accelerating their degradation, while the health evolution of the generator remains independent of the blade's state.

Maintenance optimization for UPDD systems with mixed observability introduces several significant challenges. 
First, under UPDD, the degradation evolution of the partially observable component is not governed by a fixed transition matrix but rather by a family of transition matrices associated with the observable component’s state. This greatly increases the parameter estimation burden.  The stochastic evolution of the hidden component is modulated by external random variables, 
i.e. the current state of the observable component. In other words, a modulated hidden Markov 
model is considered for explaining the functioning of the partially observable component.
Second, maintenance decision-making becomes more complex due to the need to fuse heterogeneous information sources: the fully observable component provides precise state, whereas the partially observable component requires belief-state updates that depend both on its own hidden degradation dynamics and on the state of the fully observable component. Optimizing maintenance over such a hybrid state space is inherently challenging.  Eggertsson et al. \cite{eggertsson2024maintenance} proposed a condition-based maintenance policy for capital goods with environment-dependent degradation processes, where mixed observation information is incorporated in decision-making. However, their work considers a single-component system  with observable and “irreparable”  environment, which does not align with the multi-component, repairable setting considered in this work.
Third, parameter estimation is also a challenging issue.
In the reliability domain, the expectation–maximization (EM) algorithm has been widely applied to hidden Markov models where system states are partially observable \cite{gamiz2023hidden}. Basic theoretical results concerning maximum likelihood estimation for hidden
Markov models (HMMs) were established in \cite{BaumPetrie1966}, 
\cite{Leroux:92}, and \cite{Bickel:etal:98}. The forward--backward reestimation
algorithm was introduced in \cite{Baum:etal:70}, and its
EM structure was later formalized by
\cite{Dempster:etal:77}.
However, most maintenance optimization studies under partial information have primarily focused on the maintenance decision problem, typically formulated in a partially observable Markov decision process framework (POMDPs), while assuming that all model parameters are known.
To the best of the authors’ knowledge, only \cite{gamiz2023dynamic} develops both a cost-optimized, belief-state–driven policy and a probabilistic, signal-driven preventive maintenance policy in the context of hidden Markov models.
In contrast, the present work develops a
Baum-Welch-type
algorithm to estimate the unknown parameters.

In addition, from a theoretical perspective, important  structural properties of maintenance policies such as cost monotonicity and the existence of threshold-type optimal policies can often be justified under reasonable assumptions on the transition and observation matrices, as shown in the sequential decision-making literature for maintenance models \cite{ohnishi1986optimal,byon2010season}. These structural properties for UPDD systems with mixed observability need to be  investigated.
This paper addresses these challenges by proposing a unified POMDP-based framework for maintenance optimization in two-component dependent systems with mixed observability and UPDD.

The main  contributions of this work are as follows. 
\begin{itemize}
  \item We propose a novel model considering UPDD  with mixed observability, which captures a broad range of practical engineering systems with heterogeneous sensing capabilities and unidirectional degradation influence.
  \item We investigate the structural properties of optimal maintenance policies for UPDD systems under mixed observability.
  \item 
The Baum-Welch algorithm with multiple sample paths is developed to 
estimate unknown parameters in a Hidden Markov Model, where the 
transition probabilities are determined by external information.
  \item Through numerical experiments and comparisons, we demonstrate the effectiveness of the proposed model. 
\end{itemize}
The rest of this paper is structured as follows. Section~\ref{sec_lit} reviews the literature on condition-based maintenance of dependent systems, with particular emphasis on decision-making under partial information. Section~\ref{sec_Model} describes the dependent system, its maintenance, and the associated POMDP framework. Section~\ref{sec_property} investigates the structural properties of the maintenance cost function and the optimal policy. In Section~\ref{sec_ParaEs}, a Baum-Welch algorithm with multiple sample paths  is developed to estimate the unknown system parameters. Extensive numerical experiments are presented in Section~\ref{sec_numerical} to evaluate the robustness of the parameter estimation and to assess the resulting maintenance decisions. Finally, we draw the conclusions of this work and directions for future research are discussed in Section~\ref{sec_conclusion}.

\section{Literature review}\label{sec_lit}

This paper contributes to the literature on multi-component dependent systems and condition-based maintenance decision-making with incomplete information. We begin with a review of condition-based maintenance for systems with dependent components, followed by a discussion of approaches under partial information, with a particular focus on systems composed of multiple interdependent components.
\subsection{Condition-based Maintenance of systems with dependent components}

Condition-based maintenance (CBM) has attracted increasing attention over the past decades, driven largely by advances in the Internet of Things and sensing technologies. A comprehensive review of CBM for  systems with dependent components is provided by \cite{keizer2017condition}, who systematically examined maintenance studies published between 2001 and 2018. Their review categorized system dependencies into several types, including structural dependence, economic dependence, stochastic dependence, and resource dependence; representative studies can be found in \cite{zhu2016multi,bautista2024correlation,zhang2025reliability}. More recently, Zhao et al. \cite{zhao2025condition} reviewed the literature on maintenance optimization with a particular emphasis on models developed within the framework of Markov decision processes (MDPs).
Building on these studies, the present work focuses on CBM-related research published in recent years that is most closely aligned with our scope and objectives, particularly studies addressing stochastic dependence and decision-making with MDPs.

In this regard, Oakley et al. \cite{oakley2022condition} investigated a condition-based maintenance optimization problem for a continuously monitored load-sharing system, in which a utility (reward) function was introduced to represent the value function. The impacts of both stochastic and economic dependence on the resulting policy behavior were examined. Zhao et al. \cite{zhao2023condition} developed a maintenance framework for systems with heterogeneous failure dependencies, where the failure of one component affects other components in a non-uniform and asymmetric manner. Due to the complexity introduced by such heterogeneous dependencies, no analytical structural properties of the optimal maintenance policy were derived.

Uit het Broek et al. \cite{broek2021joint} jointly optimized condition-based maintenance and production policies for a two-component load-sharing system and discussed the impact of setup costs on maintenance clustering behavior. Other related studies on condition-based maintenance with dependent components can be found in \cite{andersen2022numerical,zhang2023condition,chai2024condition}.

\subsection{Maintenance with incomplete information}

In maintenance optimization, the concept of partial information has been widely used, although its specific meaning varies across research contexts and modeling frameworks.
Generally speaking, partial information refers to situations where the true deterioration state of a system is hidden but some relevant signals can be observed and utilized for decision-making. In many studies, this corresponds to a hidden Markov process presented
in the previous subsection, in which the observation mechanism is described by a probabilistic observation matrix that links the observable outcomes to the underlying states. Different interpretations of partial information have appeared in the literature. White
\cite{white1977markov} was among the first to investigate such problems in the context of maintenance planning. He developed a general Markov quality control model in which the condition of a production process evolves as a Markov chain and is only partially observable through
noisy inspection outcomes. The decision maker must determine whether to do nothing, perform an inspection, or undertake maintenance based on these imperfect observations. This work extends the earlier study by Ross \cite{ross1971quality} and provides a detailed analysis of the
optimal policy structure when incomplete information arises from false negative/positive errors. In contrast, Tijms and Van der Duyn Schouten \cite{tijms1985markov} defined partial information as a setting in which the decision maker only knows whether the system has
failed, without access to its actual degradation level. They formulated the inspection and maintenance problem as a semi-Markov decision process with non-negligible, state-dependent maintenance times, and developed an algorithmic procedure to identify the
optimal policy that minimizes the long-run average cost rate. In this review, we focus on partial information that can be represented through a belief vector or probability distribution, which characterizes the decision maker’s uncertain knowledge of the system state. Under this representation, the partially observable Markov decision process (POMDP) framework has been widely adopted in maintenance modeling and system analysis. Ohnishi et al. \cite{ohnishi1986optimal} investigated the optimal inspection and maintenance policy for a deteriorating system in which incomplete information about the system state is periodically observed. The decision maker must determine whether to perform a perfect inspection to reveal the true state or to carry out maintenance. Their analysis extends previous results in \cite{white1977markov} to higher-dimensional belief spaces and establishes structural properties under general monotonicity conditions based on totally positive of order two ($TP_
2$) reordered matrices. Considering the environmental effects on the deterioration of wind turbine systems, Byon et al. \cite{byon2010season} developed a weather-dependent optimal maintenance policy for systems with multiple failure modes under partially observable conditions. Subsequently, Byon \cite{EunshinByon2013} extended this work by introducing additional maintenance actions and proposing a tractable approximation of the dynamic decision-making model to enable practical implementation. Similarly, Eggertsson et al. \cite{eggertsson2024maintenance} proposed a condition-based maintenance (CBM) policy for capital goods subject to environment-dependent degradation, incorporating mixed observation information into the decision-making process. In another related study, van Oosterom \cite{van2017optimal} examined an integrated inspection and sensor replacement problem in which sensor informativeness degrades over time, modeled through the Blackwell order, and demonstrated that a time-dependent threshold policy is optimal.

All the aforementioned studies primarily focus on single-component systems. In contrast, although research on multi-component systems is extensive within the maintenance optimization literature, studies that explicitly incorporate partial information remain limited. Karaba{\u{g}} et al. \cite{karabaug2020integrated, karabaug2024efficient} proposed POMDP-based models to optimize the integrated maintenance and spare parts selection problem for multi-component systems monitored by a single sensor. Similarly, Eggertsson \cite{eggertsson2025maintenance} examined a multi-component system monitored by a single sensor measuring a common condition parameter. By employing the multivariate totally positive of order two  assumption, they demonstrated that the optimal maintenance policy exhibits an at-most three-region structure when both inspection and replacement decisions are embedded within the maintenance action space.

\section{Model formulation}\label{sec_Model}
\subsection{System description}\label{sec_msystemdescriptions}
Let us consider the following notation: 

\begin{itemize} 
 
 \item Let $U_1$ denote the fully observable component of the system. Let $S_1 = 
\{0,1,\ldots,L_1\}$ denote its degradation state space where state 0 is the perfect working state 
and larger state indices represent more severe degradation. In particular, for any $j<j’ \in S_1$, 
state $j’$ is more degraded than state $j$. $L_1$ is its failure state. Its degradation evolution is 
assumed to be a homogeneous Markov chain $\{X_{1,n}\}$, with transition probability matrix 
${\bm Q}$. 
 
 \item Let $U_2$ denote the partially observable component of the system, whose 
evolution is also described by a homogeneous Markov chain, $\{X_{2,n}\}$, with state space 
$S_2 = \{0,1, \ldots, L_2\}$, where state 0 is perfect operation and $L_2$ denotes failure. State 
$i’$ is more degraded than state $i$ for $i <i’ \in S_2$. The true state of $U_2$ is not fully 
observable. Instead, we observe a signal $Z$ taking values in $O=\{0,1, \cdots, M\}$. The 
observation process is described by the observation matrix $\bm B$ , with $B(i, z)=P(Z_n=z \mid X_{2, 
n}=i), i\in S_2, z \in O$. We assume that $\{Z_n\}$ is a sequence of random variables 
conditionally independent given $\{X_{2,n}\}$. 
 
 \item The two components interact with positive dependence; for instance, when $U_1$ 
is in a a deteriorated state, $U_2$ is required to take on additional load. As a result, the state 
transition dynamics of $U_2$ is influenced by the operational state of $U_1$, with degradation 
occurring more rapidly. At each transition epoch, conditional on $U_1$ being in state $j$, the 
transition probability matrix of $U_2$ is governed by a matrix denoted ${ \bm P}^{(j)}$, $j \in 
S_1$.The entry $P^{(j)}(i,i')$ represents the probability that $U_2$ transitions to state $i’$ given 
that the current states of $U_1$ and $U_2$ are $j$ and $i$, respectively, $j\in S_1$, and $i,i’ \in 
S_2$. 
\end{itemize} 
\medskip

\subsection{Formulation of the POMDP}

Let $\bm \pi$ be the belief state of component $U_2$, denoted as 
$\bm \pi=( \pi(0),\ldots,  \pi(L_2))$, $\bm \pi \in \mathcal A=\{\bm \pi\in {\mathbb R}^{L_2+1}: \sum_{i=0}^{L_2}\pi(i)=1\}$.
At each inspection time, the decision-maker can choose maintenance actions  based on 
 $(\bm \pi, j)$, with $\bm \pi$ as the belief state of component $U_2$, $j$ is the  observable state of component $U_1$ upon the inspection.
At each inspection epoch, the decision-maker must select one of the following possible options when component $U_1$ is in the non-failed state, $i.e.$:
 \begin{itemize}
     \item $a=0$: do nothing and continue the system operation till the next inspection;
     \item $a=1$: only replace component $U_1$;
     \item $a=2$: only replace component $U_2$;
     \item $a=12$: replace the two components.   
 \end{itemize}
 Otherwise, only actions $a=1$ and $a=12$ are possible because component $U_1$ should be replaced if it is failed.
 The cost structure is as follows.
 \begin{itemize}
     \item  $\bm c_{o,k}$: the operating cost vector for component $U_k$, where the $i$th element
     $c_{o,k}(i)$ is the operating cost of component $U_k$ in state $i$,  and is assumed to be finite and non-decreasing in $i$, $i\in S_k$, $k\in \{1, 2\}$;  
     \item $c_s$: the maintenance set-up cost;
     \item $c_{r,k}$; the replacement cost of component $U_k, k\in \{1, 2\}$.
 \end{itemize}
All the above cost units are assumed to be nonnegative and finite. 
 When do-nothing is selected,  the probability of observing $(k,z)\in 
 S_1\times O$ at the next observation given current state $(\bm \pi, j)$ is
 \begin{equation}\label{Eq_Pfull}
     P( k,z\mid \bm \pi, j)
    =Q(j, k)(\bm \pi \bm P^{(j)}\bm B)_z,
 \end{equation}
  
 with the updated belief probability distribution of $U_2$ as 
 \begin{eqnarray}\label{eq_posterior}
 (T(\bm \pi, j; z))_{i'}=
\frac{\sum_{i\in  S_2} \pi(i)P^{(j)}{(i, i')} B{(i',z)}}
{(\bm \pi \bm P^{(j)}\bm B)_z}.
 \end{eqnarray}
 Notice that this expression does not depend on the state $k$ to be observed in component $U_1$. It follows from the conditional independence between the observed process of component $U_2$ and the evolution of component $U_1$.\\
 If replacement is selected,  the maintenance is performed at the end of the period with a charge of the set-up cost as well as the maintenance cost.
 Let $V(\bm \pi, j)$ be the optimal expected total discounted cost in the long-run horizon. It satisfies the following Bellman equations. \\
For $j<L_1$, meaning that  component $U_1$ is operational, then all the above four actions are accessible, and 
 \begin{eqnarray}
     V(\bm \pi, j)=
     \text{min}\big\{\Gamma^0V(\bm \pi, j),
     \Gamma^1 V(\bm \pi, j),
     \Gamma^2V(\bm \pi, j),
     \Gamma^{12} V(\bm \pi, j)\big\}.
 \end{eqnarray}

When component $U_1$ fails, it should be replaced, in which case we have that
\begin{eqnarray}
    V(\bm \pi, L_1)=\text{min}
\big\{\Gamma^1V(\bm \pi, L_1), \Gamma^{12}V(\bm \pi, L_1)\big\},
\end{eqnarray}

 where 
 \begin{eqnarray}
 \Gamma^0V(\bm \pi, j)=
 c_o(\bm \pi, j)+
 \gamma \sum_{(k,z)\in \Theta}
 P(k,z \mid \bm \pi, j)V(T(\bm \pi ,j; z), k),
 \end{eqnarray}
 
  \begin{eqnarray}\label{eq_gamma1}
 	\Gamma^1V(\bm \pi, j)=c_o(\bm \pi, j)+c_s+c_{r,1}+
 	\gamma\sum_{z \in O}(\bm \pi \bm P^{(j)}\bm B)_z V(T(\bm \pi, j; z), 0),
 \end{eqnarray}

 \begin{eqnarray}
     \Gamma^2V(\bm \pi, j)=c_o(\bm \pi, j)+c_s+c_{r,2}
     +
     \gamma \sum_{k\in S_1}
 Q(j, k)V(\bm e^0, k),
 \end{eqnarray}

 \begin{eqnarray}\label{eq_gamma12}
     \Gamma^{12}V(\bm \pi, j)=c_o(\bm \pi, j)+c_s+c_{r,1}+c_{r,2}+
     \gamma V(\bm e^0, 0). 
 \end{eqnarray}
 $c_o(\bm \pi, j)=
\sum_{i\in  S_2} \pi(i) c_{o,2}(i)+ c_{o,1}(j)$,
$\bm e^0=(1, 0, \cdots, 0) \in \mathcal A.$ $\gamma$
is the discount factor, $0<\gamma<1.$

\section{Main structural properties}\label{sec_property}

In this section, structural results for our model are derived. We first show that the value function is  concave in $\bm \pi$, the belief vector of component $U_2$.
We then provide a brief introduction to several stochastic orders—including the monotone likelihood ratio (MLR) order, first-order stochastic order, total positivity of order 2 ($TP_2$), and copositive order—which will be used as conditions to establish the structural properties of the maintenance policy. Finally, we demonstrate the monotonicity and policy structure of the model.

\subsection{Preliminaries}

\begin{theorem}
  For a fixed $j$, $V(\bm \pi, j)$ is a concave function in $\bm \pi.$
\end{theorem}

\begin{proof}
   We prove the concavity by mathematical induction. Let $V_0(\bm \pi, j)=0$ for all $(\bm \pi, j)$ initially, then the concavity holds.
   Assume that the concavity holds at iteration $r$, we will show that is holds at iteration $r+1.$
  It is easily seen that both $\Gamma^{12}V_r$ and $\Gamma^{2}V_r$ are affine functions in $\bm \pi$. Hence, they preserve the concavity.
   Let $(A(j, z))_{m\times n}=P^{(j)}(m, n)B(n, z)$ and $C(j, z)$ a column vector defining as  
  $(C(j, z))_m=\sum_{n=0}^{L_2}P^{(j)}(m, n)B(n, z).$
   Let $W_{r+1}(\bm \pi, j)=\sum_{(k, z)\in \Theta}
 P(k, z\mid \bm \pi, j)V_r(T(\bm \pi ,j; z), k),$ we will show that given $j$, $W_{r+1}(\bm \pi, j)$ is concave in $\bm \pi$.

 Let $\bm \pi=\alpha \bm \pi^1+(1-\alpha) \bm \pi^2$, $0\le \alpha\le 1$,  then
$
T(\bm \pi, j; z)=\frac{\alpha \bm \pi^1 A(j, z)+(1-\alpha) \bm \pi^2 A(j, z)}{\bm \pi C(j, z)}
    =\beta \frac{\bm \pi^1 A(j, z)}{\bm \pi^1 C(j, z)}+(1-\beta)\frac{\bm \pi^2 A(j, z)}{\bm \pi^2 C(j, z)}
    =\beta T(\bm \pi^1, j; z)+(1-\beta) T(\bm \pi^2, j; z),$
 where $\beta=\frac{\alpha \bm \pi^1 C(j, z)}{\bm \pi C(j, z)}.$
By induction, we have that
$
V_r(\bm \pi,j) \ge \beta V_r(\bm\pi^1,j) + (1-\beta)V_r(\bm\pi^2,j).
$ By multiplying $Q(j,k)(\bm \pi \bm P^{(j)}\bm B)_z$ and summing over $(k, z)$ we have 
\begin{eqnarray*}\label{eq_Wr}
W_{r+1}(\bm \pi, j)&=&\sum\limits_{(k,z)\in \Theta}
Q(j,k)(\bm \pi \bm P^{(j)}\bm B)_z V_r(T(\bm \pi, j; z),k)\\
&\ge&
\sum\limits_{(k,z)\in \Theta}
\beta Q(j,k)(\bm \pi \bm P^{(j)}\bm B)_z V_r(T(\bm \pi, j; z),k)+
(1-\beta)Q(j,k)(\bm \pi \bm P^{(j)}\bm B)_z V_r(T(\bm \pi, j; z),k).    \end{eqnarray*}
As
$
\bm\pi C(j,z)\beta = \alpha \bm \pi^1 C(j,z),$ we further have
$W_{r+1}(\bm \pi, j)\ge \alpha W_{r+1}(\bm \pi^1, j)+(1-\alpha) W_{r+1}(\bm \pi^2,j).$
Hence, $W_{r+1}(\bm \pi, j)$ is a concave function in $\bm \pi, $ implying that 
 $\Gamma^{0}V_r$ is concave of $\bm \pi$ 
 as
$
\Gamma^0 V_r(\bm\pi,j) = c_0(\bm\pi,j) + \gamma W_{r+1}(\bm\pi,j)
$
and  $c_0(\bm \pi,j)$ is linear in $\bm \pi$.  Similarly, $\Gamma^1V_r$ is concave of $\bm \pi$.
 $V_{r+1}(\bm \pi, j)$ is concave of $\bm \pi$ as 
 the minimum operator preserves concavity.
For a POMDP with finite state and observation spaces, bounded one-step costs, and a discount factor
$0<\gamma<1$, the optimal value function exists and is unique. Moreover,
 $V(\bm \pi, j)=\lim_{r\rightarrow \infty}V_r(\bm \pi, j)$ is concave in $\bm \pi$.  
\end{proof}

Before establishing the structural properties of the model, we introduce several stochastic orders that will be used later.

\begin{definition}\label{def_MLR}
For two  belief state vectors $\bm \pi^1$ and $\bm \pi^2$, 
 we say $\bm \pi^2$ dominates $\bm \pi^1$ in the 
Monotone 
Likelihood Ratio (MLR) order, noted as  
$\bm \pi^1 \le_{\text{r}} \bm \pi^2$ if $ \pi^1(i)   \pi^2(j)\geq  \pi^1(j) \pi^2(i)$
for $i<j.$
\end{definition}
A key property of MLR order is that it is closed under conditional expectations and Bayes’ rule, which makes it particularly well suited for sequential decision-making and filtering problems.
\begin{definition}\label{def_Storder}
For two belief state $1\times N$ vectors $\bm \pi^1$ and $\bm \pi^2$, 
 $\bm \pi^2$ is said to first-order
stochastically dominates
$\bm \pi^1$, noted as  
$\bm \pi^1 \le_{\text{s}} \bm \pi^2$ if 
$\sum_{i=l}^{N}  \pi^1(i) 
\leq \sum_{i=l}^{N}  \pi^2(i)$
for all $l\in  \{1, \cdots, N\}.$
\end{definition}
The first-order is a partial order on probability vectors and is weaker than the MLR order. Specifically, if
$\bm \pi^1\le_r \bm \pi^2$, then
$\bm \pi^1\le_s \bm \pi^2$.

\begin{definition}
    A stochastic matrix $\bm A$
    is  totally positive of order 2 ($TP_2$)
    if the determinant 
    $$
\begin{bmatrix}
A(i,j) & A(i,j') \\
A(i',j) & A(i',j')
\end{bmatrix}
\ge 0
$$
for all $i\le i', j\le j'$.  
\end{definition}
$TP_2$ property is widely considered in reliability and condition-based maintenance, as it captures the intuitive notion that components in more degraded states are more likely to transition to even worse states.

\begin{definition}\cite{krishnamurthy2014reduced}
  Given two $N\times N$
  stochastic matrices  $P(u)$ and $P(u+1)$, we say that $P(u) \le_c P(u+1)$ if the sequence of $N\times N$ matrices $E^{j,u}$, $j=1, \cdots, N-1$  are copositive, $i.e.$, 
  for each $j$,
\begin{equation*}
\bm \pi E^{j,u} \bm \pi' \ge 0, \quad 
\quad 
\end{equation*}
where
\begin{eqnarray*}
E^{j,u} &=& \frac{1}{2} \left[ e^{j,u}_{mn} + e^{j,u}_{nm} \right]_{N \times N},
\\
\quad
e^{j,u}_{mn} &=& P_{m,j}(u) P_{n,j+1}(u+1) - P_{m,j+1}(u) P_{n,j}(u+1).
\end{eqnarray*}
$\bm \pi$ is any belief vector in $\mathbb{R}^{1 \times N}$, 
and $\bm \pi'$ denotes its transpose. 
Moreover, $P_{m,j}(u)$ represents the $(m,j)$-th entry of $P(u)$. 
\end{definition}

Copositive order provides a partial order framework for comparing probability distributions and stochastic kernels in a way that is preserved under Bayesian updating. Intuitively, copositive order captures the idea that “larger” states reinforce one another under uncertainty: if one belief or transition kernel dominates another in this order, then the resulting posterior distributions after incorporating observations remain ordered in the same sense.
More details of copositive order can be found in \cite{krishnamurthy2014reduced}.

We make the following assumptions regarding the stochastic orders that the system parameters are required to satisfy.

Under the following  Assumption 1, conditional on the degradation state of component $U_1$, component $U_2$ in a worse degradation state is more likely to transition to even worse states in the future.
\begin{assumption}\label{assumption_P1}
$\forall j\in S_1$, $\bm P^{(j)}$ is $TP_2.$
\end{assumption}

 Assumption 2 requires a copositive order of the transition kernel of component $U_2$ across all states of component $U_1$. Intuitively, this ensures that the belief over the state of component $U_2$ preserves its order under Bayesian updating.
\begin{assumption}\label{assumption_P2}
$\forall j \in \{0, \cdots, L_1-1\}$, 
$\bm P^{(j)}\le_c \bm P^{(j+1)} $
\end{assumption}

Assumptions 3  implies that, for component $U_1$, worse degradation states are more likely to transition to even worse states.

\begin{assumption}\label{assumption_Q}
The transition probability matrix of  $\bm Q$ is $TP_2$.
\end{assumption}

Assumptions 4 implies that for component $U_2$, worse degradation states are more likely to
generate stochastically larger observations.
\begin{assumption}\label{assumption_B}
The observation probability matrix   $\bm B$ is $TP_2$.
\end{assumption}
 
Under the above stochastic orders, we can establish the following results.

\begin{lemma}\label{lemma_kernal1}
  Given $j$ and $k$, if $\bm \pi^1\le_r \bm \pi^2$, then $P( k, \cdot \mid \bm \pi^1, j)\le_r P(k, \cdot\mid \bm \pi^2, j)$. 
\end{lemma}
\begin{proof}
When $\bm P^{(j)}$ and $\bm B$ are $TP_2$, then $\bm P^{(j)}\bm B$ is $TP_2$. Hence, 
  $\bm \pi^1\bm P^{(j)}\bm B \le_r \bm \pi^2\bm P^{(j)}\bm B$ for $\bm \pi^1 \le_r \bm \pi^2,$
  implying that  $P(k, \cdot \mid \bm \pi^1, j) \le_r P(k, \cdot\mid \bm \pi^2, j)$.
\end{proof}

Lemma \ref{lemma_kernal1} presents the MLR order of the one-step conditional observation probability vector $P(\cdot, \cdot \mid \bm \pi, j )$ in Equation \eqref{Eq_Pfull}. It means that, given the same current state of component $U_1$ and its next observation, if the belief of $U_2$ is ordered according to the MLR order, then its
future observation 
will also follow an MLR order.

\begin{lemma}\label{lemma_kernal2}
Given $\bm \pi$ and $z$,  $P(\cdot, z \mid \bm \pi, j)\le_r P(\cdot,z
\mid \bm \pi, l)$ for $j\le l \in S_1$. 
\end{lemma}
\begin{proof}
  Given  $\bm \pi \in \mathcal A$ and $z\in O$, for $j\le l\in S_1$ and $m\leq n \in S_1$, we have that
  \begin{eqnarray}
\nonumber && \begin{bmatrix}
Q(j, m)(\bm \pi \bm P^{(j)}\bm B)_z  & Q(j, n)(\bm \pi \bm P^{(j)}\bm B)_z\\
Q(l, m)(\bm \pi \bm P^{(l)}\bm B)_z  & Q(l, n)(\bm \pi \bm P^{(l)}\bm B)_z\\
\end{bmatrix}\\\notag
&&=
  \begin{bmatrix}
Q(j, m) & Q(j, n)\\
Q(l, m)& Q(l, n)\\
\end{bmatrix}(\bm \pi \bm P^{(j)}\bm B)_z(\bm \pi \bm P^{(l)}\bm B)_z
\geq 0.
\label{eq:conditiaoncross}\\
&& 
\end{eqnarray}
Hence, the matrix
$$
\begin{bmatrix}
P(\cdot, z \mid \bm \pi, j) \\
P(\cdot,z
\mid \bm \pi, l)
\end{bmatrix}
$$
is $TP_2$. From the
Definition 4.2 in  \cite{ohnishi1986optimal},
we conclude that 
$P(\cdot,z
\mid \bm \pi, j)\le_r
P(\cdot,z
\mid \bm \pi, l)
$.
 \end{proof}

 Lemma \ref{lemma_kernal2}
 shows that given the current belief vector of $U_2$ and the next observation of $U_2$,
 if the system is currently in a less degraded state for $U_1$, then
 the future observation of $U_1$ is expected to be smaller in the MLR order.
 
\begin{lemma}\label{lemma_kernel3}
 Given $\bm \pi$ and $k$,  $P(k, \cdot  \mid \bm \pi, j)\le_r P(k, \cdot \mid \bm \pi, l)$ for $j\le l \in S_1$.
 \end{lemma}
 \begin{proof}
 For $z\le \tilde z \in O$, we have 
  \begin{eqnarray}\label{eq_prooflemma1}
 \begin{bmatrix}
Q(j, k)((\bm \pi \bm P^{(j)} \bm B)_z  & Q(j, k)(\bm \pi \bm P^{(j)}\bm B)_{\tilde z}\\
Q{(l, k)}((\bm \pi \bm  P^{(l)} \bm B)_z  & Q(l, k)(\bm \pi \bm P^{(l)} \bm B)_{\tilde z}\\
\end{bmatrix}=Q(j,k)Q(l, k)
  \begin{bmatrix}
\bm (\bm \pi \bm P^{(j)}\bm B)_z  & (\bm \pi \bm P^{(j)}\bm B)_{\tilde z}\\
\bm (\bm \pi \bm P^{(l)}\bm B)_z  & (\bm \pi \bm P^{(l)}\bm B)_{\tilde z}\\
\end{bmatrix}.
\end{eqnarray}
With Assumption \ref{assumption_P2}, we know that $\bm \pi \bm P^{(j)}\le_r \bm \pi \bm P^{(l)}$,
as $\bm B$ is $TP_2$, hence  $\bm \pi \bm P^{(j)}\bm B\le_r \bm \pi \bm P^{(l)}\bm B$, implying that Equation \eqref{eq_prooflemma1} is non-negative.
\end{proof}

Lemma \ref{lemma_kernel3} shows that given the same current belief of $U_2$
and the future state of $U_1$, if the system is currently in a less degraded state for $U_1$,
then
the future observation of $U_2$ is also expected to be smaller in the MLR order. The following lemmas examine the properties of the posterior belief probability distribution of $U_2$ given in Equation \eqref{eq_posterior}.

\begin{lemma}\label{lemma_beliefz}
  Given $\bm \pi$ and $j$, $T(\bm \pi, j; z)\le_r T(\bm \pi, j; \tilde z)$ for $z\le \tilde z \in O$.
\end{lemma}

\begin{proof}
  Let $\bm {\tilde \pi}=\bm \pi\bm P^{(j)}$, then for $i \le l \in S_2$, 
  
   \begin{eqnarray}\label{eq_prooflemma}
 \begin{bmatrix}
T(\bm \pi, j; z)_i  &T(\bm \pi, j; z)_l\\
T(\bm \pi, j; \tilde z)_i  &T(\bm \pi, j; \tilde z)_l\\
\end{bmatrix}=
\frac{\tilde {\bm \pi}(i) \tilde {\bm \pi}(l)}
{(\tilde {\bm \pi} \bm B)_z (\tilde {\bm \pi} \bm B)_{\tilde z}}
  \begin{bmatrix}
 B({i, z})  & B({l, z})\\
 B({i, \tilde z})  & B({l, \tilde z})\\
\end{bmatrix}
\ge 0.
\end{eqnarray}
\end{proof}

Lemma \ref{lemma_beliefz} establishes that given the same current system state, the updated belief distribution is increasing in the MLR order with respect to future signal of $U_1$. 
  
 Lemma \ref{lemma_beliefpi} shows that, conditional on the current state of $U_1$
 and an observation of $U_2$ at the next transition, the posterior belief distribution of $U_2$
 is increasing in the MLR order as a function of the prior belief.
 
 \begin{lemma}\label{lemma_beliefpi}
   If $\bm \pi^1\le_r \bm \pi^2$, then $T(\bm \pi^1, j; z)\le_r T(\bm \pi^2, j; z)$, $(j, z)\in S_1 \times O.$
 \end{lemma} 
  
  \begin{proof}
    Let $\bm \pi^1\le_r \bm \pi^2$, 
    $ {\tilde {\bm\pi}^{u}}=\bm \pi\bm^{u} P^{(j)}$, $u=1,2, j\in S_1.$
Then for 
    $i<l \in S_2$, we have 
\begin{eqnarray}\label{eq_prooflemma}
 \begin{bmatrix}
T(\bm \pi^1, j; z)_i  &T(\bm \pi^1, j; z)_l\\
T(\bm \pi^2, j; z)_i  &T(\bm \pi^2, j; z)_l\\
\end{bmatrix}=
\frac{B(i, z) B (l, z)}
{(\tilde {\bm \pi}^1 \bm B)_z(\tilde {\bm \pi}^2 \bm B)_z}
  \begin{bmatrix}
 \tilde { \pi}^1(i)  & \tilde { \pi}^1(l)\\
\tilde { \pi}^2(i) & \tilde { \pi}^2(l)\\
\end{bmatrix}
\ge 0.
\end{eqnarray}
    \end{proof}

Lemma \ref{lemma_beliefj} indicates that given the current belief $\bm \pi$ and an observation of $U_2$
at the next transition, the updated belief distribution is MLR-increasing with respect to the current state of $U_1$.

\begin{lemma}\label{lemma_beliefj}
   $T(\bm \pi, j; z)\le_r T(\bm \pi, j^+; z) $ for $j\le j^+ \in S_1.$
 \end{lemma}
 
\begin{proof}
For $j<j^+\in S_1$ and $i<l\in S_2,$   \begin{eqnarray}\label{eq_prooflemma}
 \begin{bmatrix}
T(\bm \pi, j; z)_i  &T(\bm \pi, j; z)_l\\
T(\bm \pi, j^+; z)_i  &T(\bm \pi, j^+;  z)_l\\
\end{bmatrix}=
\frac{B(i, z)B(l, z)}
{({\bm \pi}P^{(j)}\bm B)_z({\bm \pi} P^{(j^+)}\bm B)_z}
  \begin{bmatrix}
{(\bm \pi \bm P^{(j)})}_i  &  {(\bm \pi \bm P^{(j)})}_l \\
{(\bm \pi \bm P^{(j^+)})}_i  &  {(\bm \pi \bm P^{(j^+)})}_l \\
\end{bmatrix}
\ge 0.
\end{eqnarray}
\end{proof}

%
%
  \subsection{Value function properties}

 In the following, we establish the structural properties of the value function and the maintenance policy based on the above lemmas.
 Theorem \ref{thm_2} shows the monotonicity of the value function with respect to $\bm \pi$, the belief distribution of $U_2$ and the state of $U_1$ respectively.
\begin{theorem}\label{thm_2}
 The value function $V(\bm \pi, j)$ has the following properties:
  \begin{itemize}
   \item [(i)] For $\bm \pi^1\le_r \bm \pi^2 \in \mathcal A$, $V(\bm \pi^1, j)\leq V(\bm \pi^2, j)$, $\forall j\in  S_1$.
   \item [(ii)] For $j\leq j'\in   S_1,$
   $V(\bm \pi, j)\le V(\bm \pi, j'),$  $\forall \bm \pi \in \mathcal A.$
  \end{itemize}
\end{theorem}
\begin{proof}
The value iteration is implemented to show the monotonicity of the value function.
We utilize a subscript to label each iteration. Suppose that $V_0(\bm \pi, j)=0$
for all $(\bm \pi, j) \in \mathcal A\times S_1.$ 

We first prove (i), the monotonicity of $V$ with respect to $\bm \pi.$
Assume that at iteration $r$, 
$V_r$ is non-decreasing in $j$ for any 
$\boldsymbol{\pi} \in \mathcal{A}$, and MLR-increasing in $\boldsymbol{\pi}$ 
for any $j \in S_1$.
It is easily seen that $\Gamma^{2}V_r(\bm \pi, j)$ and 
$\Gamma^{12}V_r(\bm \pi, j)$ hold the monotonicity with $\bm \pi$.

We then prove that $\Gamma^1 V_r(\bm \pi, j)$ is MLR-increasing in $\bm \pi.$  
From Lemma \ref{lemma_beliefz},
given $(\bm \pi, j),$
$T(\bm \pi, j; z)$ is MLR-increasing with $z$. Hence,
$V_r(\cdot)$ is increasing in $z$. As
$\bm \pi^1 \bm P^{(j)}\bm B \le \bm \pi^2 \bm P^{(j)}\bm B$ for $\bm \pi^1\le_r \bm \pi^2$,
from Lemma 4.1 in \cite{ohnishi1986optimal},
we have that 
$\sum_{z\in O}(\bm \pi^1 \bm P^{(j)}\bm B)_z V_r(T(\bm \pi^1, j; z), 0) \le  \sum_{z\in O}(\bm \pi^2 \bm P^{(j)}\bm B)_z V_r(T(\bm \pi^1, j; z), 0).$
From Lemma \ref{lemma_beliefpi}, we know that $T(\bm \pi^1, j; z)\le_r T(\bm \pi^2, j; z)$.
Hence $V_r(T(\bm \pi^1, j; z), 0)\le V_r(T(\bm \pi^2, j; z), 0)$, implying that 
$$\sum_{z\in O}(\bm \pi^1 \bm P^{(j)}\bm B)_z V_r(T(\bm \pi^1, j; z), 0) \le  \sum_{z\in O}(\bm \pi^2 \bm P^{(j)}\bm B)_z V_r(T(\bm \pi^2, j; z), 0).$$
Therefore, $\Gamma^1V_r(\bm \pi, j)$ is MLR-increasing with $\bm \pi$.

For $\Gamma^{0}V_r(\bm \pi, j)$, similar to the proof of $\Gamma^1V_r(\bm \pi, j)$, we know that given $k$,
$$\sum_{z\in O}(\bm \pi^1 \bm P^{(j)}\bm B)_z V_r(T(\bm \pi^1, j; z), k) \le  \sum_{z\in O}(\bm \pi^2\bm  P^{(j)}\bm B)_z V_r(T(\bm \pi^2, j; z), k).$$
Hence, $$ \sum_{k\in S_1}Q(j,k)\sum_{z\in O}(\bm \pi^1 \bm P^{(j)}\bm B)_z V_r(T(\bm \pi^1, j; z), k)
\le \sum_{k\in S_1}Q(j,k) \sum_{z\in O}(\bm \pi^2 \bm P^{(j)}\bm B)_z V_r(T(\bm \pi^2, j; z), k),$$
implying that $\Gamma^{0}V_r(\bm \pi ,j)$ is MLR-increasing with $\bm \pi.$
Therefore, $V_{r+1}$ is MLR-increasing with $\bm \pi.$

Next, we prove the monotonicity of  $V$ with $j$ in (ii).
$\Gamma^{12}V_r$ is constant with respect to $j$.
For $\Gamma^2V_r$, it is non-decreasing with $j$ as  $\bm Q$ is $TP_2.$
For $\Gamma^1V_r$, we know that $\bm \pi \bm P^{(j)}\bm B \le_r \bm \pi \bm P^{(j^+)}\bm B$, hence 
$$\sum_{z\in O}(\bm \pi \bm P^{(j)}\bm B)_z V_r(T(\bm \pi, j; z), 0) \le  \sum_{z\in O}(\bm \pi \bm P^{(j^+)}\bm B)_z V_r(T(\bm \pi, j; z), 0),$$
as $T(\bm \pi, j; z)$ is MLR-increasing with $z$ and $V_r$ is increasing with $z$.
From Lemma \ref{lemma_beliefj}, it is seen that $T(\bm \pi, j; z)\le_r T(\bm \pi, j^+; z).$
Hence, $V_r(T(\bm \pi, j; z), 0)\le V_r(T(\bm \pi, j^+; z) ,0)$.
We therefore have 
$$\sum_{z \in O}(\bm \pi \bm P^{(j)}\bm B)_z V_r(T(\bm \pi, j; z), 0) \le  \sum_{z\in O}(\bm \pi \bm P^{(j^+)}\bm B)_z V_r(T(\bm \pi, j^{+}; z), 0),$$
indicating that $\Gamma^1 V_r$ is non-decreasing with $j.$

For $\Gamma^{0}V_r,$ similarly, with a fixed $k$, 
$$\sum_{z\in O}(\bm \pi \bm P^{(j)}\bm B)_z V_r(T(\bm \pi, j; z), k) \le  \sum_{z\in O}(\bm \pi \bm P^{(j^+)}\bm B)_z V_r(T(\bm \pi, j^{+}; z), k),$$
Hence, 
\begin{align*}
\sum_{k\in S_1} Q(j,k)
\sum_{z\in O}
\bigl(\boldsymbol{\pi}\,\boldsymbol{P}^{(j)}\boldsymbol{B}\bigr)_{z }
\,V_r\!\left(T(\boldsymbol{\pi}, j; z),\, k\right)
\;\le\;
\sum_{k\in S_1} Q(j,k)
\sum_{z\in O}
\bigl(\boldsymbol{\pi}\,\boldsymbol{P}^{(j^+)}\boldsymbol{B}\bigr)_{z}
\,V_r\!\left(T(\boldsymbol{\pi}, j^{+}; z),\, k\right).
\end{align*}

 $\Gamma^0V$ is thus  non-decreasing of $j.$
Hence, $V_{r+1}$ is non-decreasing with $j$ as the minimization operator preserves the monotonicity.
\end{proof}

\subsection{Maintenance policy properties}
Theorems \ref{thm_3} and \ref{thm_4} provide some structural properties of the maintenance policy. 
\begin{theorem}\label{thm_3}
 The optimal maintenance policy has the following properties.
     
    \begin{itemize}
      \item [(i)] There exists a threshold $l^*$ for the observable state of $U_1$, 
which is independent of $\bm \pi$, 
such that action $a=12$ is preferred over action $a=2$ when $j > l^*$, 
while action $a=2$ is preferred over action $a=12$ otherwise.

    
  \item [(ii)]There exist boundaries in MLR order $\bm \pi^{*}_{0,12}(j)$, $\bm \pi^{*}_{1,12}(j)$, $\bm \pi^{*}_{0,2}(j)$
  and $\bm \pi^{*}_{1, 2}(j)$  
  between action pairs $(0, 12)$, $(1, 12)$, $(0, 2)$
  and $(1, 2)$.  The first three boundaries are non-increasing functions of the observable state $j, j\in  S_1$.
  
    
    \end{itemize}

\end{theorem}

\begin{proof}
From the Bellman equation, it is seen that  $\forall (\bm \pi, j)$, $\Gamma^2V(\bm \pi, j)-\Gamma^{12}V(\bm \pi, j)$
is a non-decreasing function of $j$ and is independent of $\bm \pi.$ Hence, there exists a boundary $l^*$ such that $\Gamma^2V(\bm \pi, j)-\Gamma^{12}V(\bm \pi, j)<0$ for $j\le l^*$ whereas  $\Gamma^2V(\bm \pi, j)-\Gamma^{12}V(\bm \pi, j)\ge 0$ otherwise. Hence, the first statement (i) is justified. 

In the next, we will show the existence of the threshold $\bm \pi^*_{0,12}(j)$ and its monotonicity with $j$ in (ii).
From Theorem \ref{thm_2} and the Bellman equation,
we know that $\Gamma^0V(\bm \pi, j)-\Gamma^{12}V(\bm \pi, j)$ is MLR-increasing with $\bm \pi$. Hence, there exists a boundary $\bm \pi_{0,12}^*(j)$
such that action $a=12$ is preferred over action $a=0$ when $\bm \pi \ge_r \bm \pi_{0,12}^*(j)$, whereas action $a=0$ is preferred otherwise.
In addition, $\Gamma^0V(\bm \pi, j)-\Gamma^{12}V(\bm \pi, j)$ is increasing in $j$. Hence,
given a belief $ \tilde {\bm\pi}\ge_r \bm \pi_{0,12}^*(j)$,
we know that $\Gamma^0V(\tilde {\bm\pi}, j+1)- \Gamma^{12}V(\tilde {\bm\pi}, j+1)
\ge \Gamma^0V(\tilde {\bm\pi}, j)- \Gamma^{12}V(\tilde {\bm\pi}, j) \ge 0,$
 indicating that $\bm \pi_{0,12}^*(j)\ge_r \bm \pi_{0,12}^*(j+1)$.
 The establishment of other boundaries can be validated in a similar way and we omit the proof here.
\end{proof}

%

\begin{theorem}\label{thm_4}
 The optimal maintenance policy induces the following properties.
  \begin{itemize}
    \item [(i)] If the optimal maintenance policy in state $(\bm \pi, j)$ is to replace both components ($a=12$), then the same action remains optimal for any state  $(\bm \pi^+, j^+)$ such that $\bm \pi \le_r \bm \pi^+ \in \mathcal A$ and $j\le j^+ \in S_1$.
    
    \item [(ii)]For any $j\in S_1,$ if it is optimal to only replace component $U_2 \;(a=2)$ in state $(\bm \pi, j)$,  then it will  be optimal to either only replace component $U_2 \;(a=2)$ or to replace both the two components ($a=12$)
    for system with
    state $(\bm \pi^+, j)$, $\bm \pi \le_r \bm \pi^+
    \in \mathcal A$.
        
     \item  [(iii)] For any   $j \in S_1$, the optimal regions for the action “Only replace Component $U_1$" ($a=1$) and the action “Replace Components $U_1$ and $U_2
     $" ($a=12$) are mutually exclusive with respect to the belief  $\bm \pi, \bm \pi \in \mathcal A$.
  \end{itemize}
\end{theorem}

\begin{proof}
    We give a brief proof of Theorem \ref{thm_4}.
    Statement (i) follows from the fact that for any 
$(\boldsymbol{\pi}, j) \in \mathcal{A} \times S_1$, 
$\Gamma^{l}V(\boldsymbol{\pi}, j) - \Gamma^{12}V(\boldsymbol{\pi}, j)$
is non-decreasing in $\boldsymbol{\pi}$ and in $j$, respectively, for all
$l \in \{0,1,2\}$.
Since
$\Gamma^{l}V(\boldsymbol{\pi}, j) - \Gamma^{2}V(\boldsymbol{\pi}, j)$
is non-decreasing in $\boldsymbol{\pi}$ in the MLR order for $l \in \{0,1\}$, statement (ii) follows.
Statement (iii) holds because
$\Gamma^{2}V(\boldsymbol{\pi}, j) - \Gamma^{12}V(\boldsymbol{\pi}, j)$
is independent of $\boldsymbol{\pi}$.
\end{proof}

The structural results above arise from the positive degradation coupling between the two components and the asymmetric information structure. Since $U_1$ is fully observable and a worse state of $U_1$ accelerates the degradation of the partially observable  $U_2$, both the expected future cost and the uncertainty about $U_2$ increase monotonically. Consequently, maintenance actions become more aggressive as either the observable state worsens or the belief about $U_2$ deteriorates. This leads to threshold-type decisions in the observable state and monotone switching boundaries in the belief under the MLR order. Moreover, because a worse observable state amplifies the future deterioration of $U_2$, the belief thresholds separating different actions decrease with the observable state, resulting in mutually exclusive optimal maintenance regions. Another noteworthy observation is that no threshold boundary exists between actions 0 and 1, indicating that in multi-component dependent systems the optimal policy structure is more complex and less intuitive than in single-component settings.

\section{Parameter estimation}\label{sec_ParaEs}
\subsection{Preliminary notes}
In the absence of maintenance intervention,
the model describing the system evolution can be described by the triplet $(X_1,X_2,Z)$, where $X_1$ is an observable Markov chain associated with component $U_1$ and the pair $(X_2,Z)$ forms an HMM for component $U_2$ whose distribution depends on the state of $U_1$. Assuming that both components start in perfect condition, the unknown parameters of the model consist of $L_1+3$ matrices $\theta=( {\bm Q},{\bm P}^{(0)},\ldots, {\bm P}^{(L_1)},{\bm B})$, and the total number of independent parameters to be estimated is $(L_1+1)L_1+(L_1+1)(L_2+1)L_2+(L_2+1)(M-1)$.

To estimate the model parameters, we use the expectation-maximization (EM) algorithm, treating the hidden state of $U_2$ as missing data, while the observed data consist of the successive emission signals associated with the state of component $U_2$ and the successive states of component $U_1$.

Since the behaviour of component $U_1$ does not depend on the state of component $U_2$ and it is fully observable, we can estimate the transition matrix of $X_1$, $\bm Q$, directly from the available sample information from $X_1$. Then, to estimate the matrices ${\bm P}^{(\cdot)}$ and $\bm B$, we  will formulate a maximum likelihood estimation problem conditional on the observations on $X_1$.

The EM framework is particularly well suited to this setting, as it allows parameter estimation to be decomposed into an E-step, in which the posterior distributions of the hidden states are computed via the Baum-Welch algorithm (forward-backward algorithm), and an M-step, in which the parameters are updated by maximizing the expected complete-data log-likelihood. In \cite{Bickel:etal:98},  consistency of the maximum likelihood estimator is proved under the assumption that the hidden Markov chain is ergodic. In our setting, however, the hidden process $X_2$ contains an absorbing state and is therefore
non-ergodic, so these results do not apply directly. For non-ergodic hidden
Markov chains, \cite{Douc:etal:2010} established stability (forgetting of the initial
distribution) of the filtering distribution under suitable conditions. In the present model, the hidden process has a unique absorbing failure state
and an irreversible degradation structure. This structure ensures forgetting of
the initial condition of the filter (forward-backward system) in the sense of \cite{Douc:etal:2010}. However, it also implies that along any single trajectory each irreversible transition (for example, from state 0 to 1 or from 1 to 2) can occur at most once. As a result, a single realization does not allow sufficient information to
accumulate for consistent estimation of all transition probabilities.

Consistency of the MLE requires that the normalized observed log-likelihood
converges to its expected value, which in turn relies on sufficient accumulation
of information for all model parameters. In our model, this accumulation cannot
be achieved along a single trajectory due to the absorbing structure. To
overcome this limitation, we adopt a multi-trajectory framework in which
multiple independent realizations of the HMM are observed. This allows the
effective sample size associated with each transition to grow, yielding stable
and statistically meaningful parameter estimates via the EM algorithm. In
addition, aggregating information across independent trajectories increases the
total Fisher information and consequently reduces the variance of the resulting
parameter estimators.

\paragraph{Final Remarks} 
\begin{enumerate}
\item[(i)] The algorithm we use accounts for the heterogeneous structure of the transition probabilities of the hidden component, which depend on the current state of the observable component. This requires a slight modification of the standard Baum-Welch algorithm.  

\item[(ii)] The properties of the maintenance model developed in the previous sections rely on certain structural conditions on the model matrices. These conditions are not explicitly enforced during the parameter estimation process, so the theoretical guarantees for the maintenance policies based on the estimated parameters are not formally ensured. However, since our estimators possess strong consistency properties, we can still empirically validate that the maintenance policies perform, which will be presented in Section \ref{sec_6.1} in more detail.
\end{enumerate}

\subsection{The Baum-Welch algorithm with multiple sample paths}
Suppose that we observe $T$ independent trajectories of size $n$ each, that is, let us consider the dataset  ${\mathcal X}=\left\{x_{1,0:n}^{(t)},z_{1:n}^{(t)}\right\}_{t=1}^T$, 
where we use notation $x_{1,0:n}^{(t)}=(x_{1,0}^{(t)}, x_{1,1}^{(t)},\ldots, x_{1,n}^{(t)})$ and $z_{1:n}^{(t)}=(z_{1}^{(t)},\ldots, z_n^{(t)})$, for $t=1,\ldots, T$.

By independence across trajectories, the observed-data log-likelihood is given by
\begin{eqnarray}\label{eq:likelihood}
\nonumber \ell(\theta)&=&\sum_{t=1}^T \log \PP_{\theta}(Z_{1:n}^{(t)}=z_{1:n}^{(t)}\mid X_{1,0:n}^{(t)}=x_{1,0:n}^{(t)})+\sum_{t=1}^T \log \PP_{\theta}(X_{1,0:n}^{(t)}=x_{1,0:n}^{(t)})\\
&\equiv&
\ell_1(\theta)+\ell_2(\theta).
\end{eqnarray}

We can obtain maximum likelihood estimators for the entries of matrix $\bm Q$, which governs the transitions between states of component $U_1$, directly optimizing $\ell_2$ in expression \eqref{eq:likelihood}, that is,
\[
\widehat{Q}(j,j')=\frac{\sum_{t=1}^T\sum_{k=1}^n {\bf 1}_{\{X_{1,k-1}^{(t)}=j,X_{1,k}^{(t)}=j'\}}}{\sum_{t=1}^T\sum_{k=1}^n {\bf 1}_{\{X_{1,k-1}^{(t)}=j\}}}, \quad j, j' \in S_1. 
\]

For estimating the parameters of the HMM part of the model, we consider the conditional likelihood given the observed states of component $U_1$ along the $T$ samples. Let us introduce the following notations for the posterior marginal and pairwise transition probabilities, respectively
\begin{equation}\label{eq:phi} 
	\phi_{k}^{(t)}(i)= \PP_{\theta}(X_{2,k}^{(t)}=i \mid \mathcal{X}), k=1,\cdots,n;
\end{equation}
and
\begin{equation}\label{eq:xi}
	\xi_{k}^{(t)}(i,i') = \PP_{\theta}(X_{2,k-1}^{(t)}=i, X_{2,k}^{(t)}=i' \mid \mathcal{X})  
\end{equation}
for $t=1,\ldots, T$; $k=1,\ldots, n$; and, $i,i' \in S_2$.

\subsection*{M-step}
\noindent At   iteration $r$, we have an estimate $\theta_{[r]}$ and,  for trajectory $t$, we define 
$\phi_{[r];k}^{(t)}(i)$ and $\xi_{[r];k}^{(t)}(i,i')$ 
similar to Equations \eqref{eq:phi}-\eqref{eq:xi}, respectively. Then we obtain the updated following estimation of the parameters
\begin{eqnarray}\label{eq_Pj}
{P}^{(j)}_{[r+1]}(i,i')=\frac{\sum_{t=1}^T\sum_{k=1}^n \xi_{[r];k}^{(t)}(i,i') {\bf 1}_{\{X_{1,k-1}^{(t)}=j\}}}
{\sum_{t=1}^T\sum_{k=1}^n \phi_{[r];k-1}^{(t)}(i) {\bf 1}_{\{X_{1,k-1}^{(t)}=j\}}}
\end{eqnarray}
with 
$\phi_{[r];0}^{(t)}(i)=\delta_{i,0}$.
$i,i' \in S_2$ and 
$j \in S_1$.

\[
{B}_{[r+1]}(i,z)= \frac{\sum_{t=1}^T\sum_{k=1}^{n} \mathbf{1}_{\left\{Z_k^{(t)}=z\right\}} \phi_{[r];k}^{(t)}(i)}{\sum_{t=1}^T\sum_{k=1}^{n} \phi_{[r];k}^{(t)}(i)},
\]
for $i\in S_2$ and $z \in O$.

\subsection*{E-step}
\noindent At iteration $r$, the E-step requires the computation of the posterior probabilities $\phi_{[r];k}^{(t)}(i)$ and $\xi_{[r];k}^{(t)}(i,i')$ for $i,i' \in S_2$, $k=1,\dots,n$ and $t=1,\ldots, T$. To do it, we define the forward probabilities
\[
\alpha^{(t)}_{k}(i)=\PP_{\theta}\big(X_{2,k}^{(t)}=i,\,Z_{1:k}^{(t)}=z_{1:k}^{(t)}\mid X_{1,1:n}^{(t)}=x_{1,1:n}^{(t)}\big),
\qquad i\in S_2,
\]
with initialization 
\[
\alpha^{(t)}_{1}(i)= B(i,z_1^{(t)})P^{(0)}(0,i), 
\]
given that $\PP_{\theta}(X_{1,0}=0)=1$; and  $\PP_{\theta}(X_{2,0}=0)=1$. The forward recursion is then
\[
\alpha^{(t)}_{k}(i)=\sum_{i'\in S_2}\alpha^{(t)}_{k-1}(i')\,P^{(x_{1,k-1}^{(t)})}(i',i)\,B(i,z_k^{(t)}), k=2,\cdots,n.
\]

Similarly, we define the backward probabilities
\[
\beta^{(t)}_k(i)=\PP_{\theta}\big(Z_{(k+1):n}^{(t)}=z_{(k+1):n}^{(t)} \mid X_{2,k}^{(t)}=i,X_{1,n}^{(t)}=x_{1,1:n}^{(t)}\big), \quad i \in S_2,
\]
for all $k=1,\ldots, n-1$, and $t=1,\ldots, T$. With terminal condition $\beta^{(t)}_n(i)=1$ for all $i\in S_2$;  and, the recursion is then
\[
\beta^{(t)}_k(i)=\sum_{i'\in S_2} P^{(x_{1,k}^{(t)})}(i,i')\,B(i',z_{k+1}^{(t)})\,\beta^{(t)}_{k+1}(i').
\]

At the $r$ iteration, we have an estimate $\theta_{[r]}$, and the forward-backward algorithm is applied separately to each trajectory. For trajectory $t$, we have $\alpha_{[r],k}^{(t)}(i)$ and $\beta_{[r],k}^{(t)}(i)$, and the required posterior probabilities are obtained as
\[
\phi_{[r],k}^{(t)}(i)=\frac{\alpha_{[r],k}^{(t)}(i)\,\beta_{[r],k}^{(t)}(i)}
{{\cal L}^{(t)}(\theta_{[r]})},
\]
and
\[
\xi_{[r],k}^{(t)}(i,i')=
\frac{\alpha_{[r],k-1}^{(t)}(i)\,P_{[r]}^{x_{1,k-1}^{(t)}}(i,i')\,B_{[r]}(i',z_k)\,\beta_{[r],k}^{(t)}(i')}
{{\cal L}^{(t)}(\theta_{[r]})},
\]
where 
$${\cal L}^{(t)}(\theta_{[r]})={\sum_{i'\in S_2} 
\alpha_{[r],k}^{(t)}
(i')\,\beta_{[r],k}^{(t)}(i')},$$ 
 is the likelihood of the $t$ sample trajectory, for $t=1,\ldots, T$.

The E-step and M-step are iterated until convergence.

\section{Numerical Experiments}\label{sec_numerical}
Next, we first give a motivating example to show the applicability scenario of the proposed model and examine the characteristics of the proposed policy using the estimated system parameters. We show that, for the considered small-scale problem, the proposed EM algorithm yields consistent and robust parameter estimates. When the estimated parameters are used in place of the true parameters to compute the optimal maintenance policy, the resulting policies exhibit very similar structural properties, and the associated cost difference is very small.
We then introduce three alternative commonly-used maintenance policies for comparison. Through numerical experiments involving 64 problem instances, we demonstrate that the proposed policy consistently outperforms these alternatives across all cases considered.
Finally, we conduct a series of sensitivity analyses to investigate how the optimal policy responds to variations in key system parameters, thereby providing further insights into the policy’s structural characteristics and practical behavior.

\subsection{A motivating example}\label{sec_6.1}

Consider a rotating drivetrain typical of pumps, compressors, conveyors, or industrial gearboxes. Component $U_1$ is a rolling-element bearing supporting the rotor/shaft. Its degradation is generally measured by lubrication starvation, contamination, $etc.$, which is often directly  monitored.
As the bearing degrades, progressive wear and damage can lead to a loss of stiffness and alignment at the support, causing misalignment, which provides non-ideal working condition for the shaft/gear (component $U_2$).
Consequently, the degradation of $U_2$ becomes more severe and faster. This creates UPDD: when the bearing is in a “worse” health state, the transition probabilities for the shaft/gear moving from healthy to failure increase.
In practice, the gear/shaft is often partially observable, such that its health is inferred from indirect indicators such as torque ripple or oil debris.

Assume that the rolling-element bearing consists of three states, $i.e.,$ $S_1=\{0, 1, 2\}$, representing the perfect working, degraded and failed states correspondingly. The transition matrix is 
$$
\bm Q=\begin{bmatrix}
    0.8&0.2& 0\\
    0& 0.7& 0.3\\
    0& 0&1
\end{bmatrix}$$

The state of the gear is classified into working and failure states,
$i.e.,$ $S_2=\{0, 1\}.$
Given the state of $U_1$, the transition matrices of $U_2$ are as follows.
$$\bm P^{(0)}
=\begin{bmatrix}
    0.8 & 0.2\\
    0& 1
\end{bmatrix},$$
$$\bm P^{(1)}
=\begin{bmatrix}
    0.7 & 0.3\\
    0& 1
\end{bmatrix},$$
$$\bm P^{(2)}
=\begin{bmatrix}
    0.5 & 0.5\\
    0& 1
\end{bmatrix}.$$

Its observation probability matrix is 
$$\bm B=\begin{bmatrix}
    0.8& 0.2\\
    0.2 &0.8
\end{bmatrix}.$$

\begin{figure}
        \centering \includegraphics[width=0.9\linewidth]{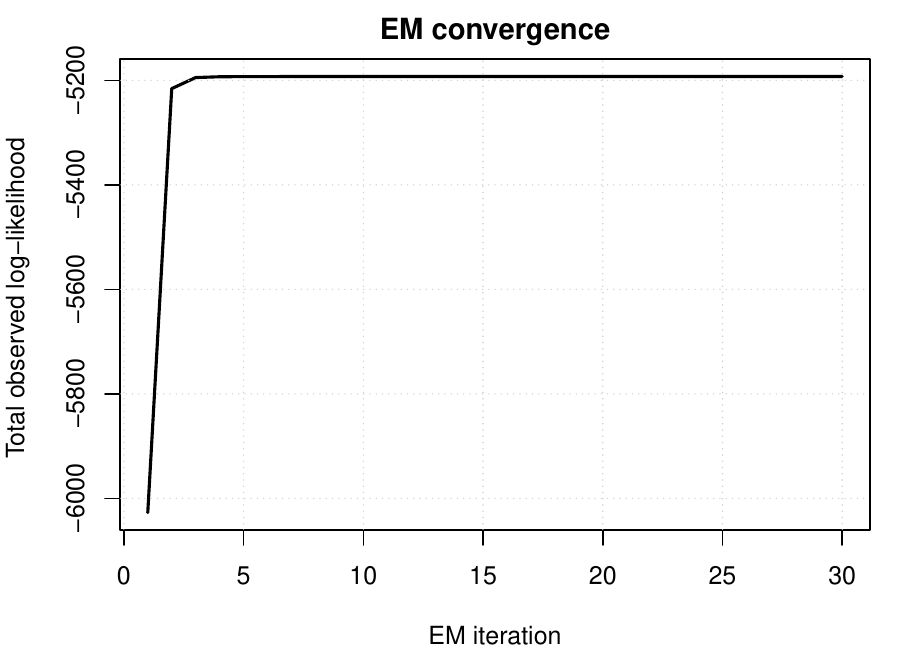}
        \caption{Observed Log-Likelihood Versus Iteration} \label{fig:EM}
\end{figure}

We generate $T= 200$ independent trajectories, with length $n=50$ for each trajectory.
In the update of  $\bm P^{(j)}$  matrices presented in Equation \eqref{eq_Pj},
if the second row of $\bm P^{(j)}$ corresponds to an absorbing state,
 the update is applied only to the first row. The maximal EM iteration number  is 30. 
 Figure \ref{fig:EM} shows the convergence of the EM algorithm, where the log-likelihood increases monotonically and stabilizes after only a small number of iterations, demonstrating the fast convergence property of the EM algorithm for this problem. 
 Table~\ref{tab:robustness_params} reports the mean and standard deviation of the parameter estimates obtained from the EM algorithm under 30 random initializations. Across all parameters, the estimated means are highly stable and the corresponding standard deviations are small, suggesting that the estimation procedure is robust to initialization and yields consistent estimates. In particular, the  estimates of $\bm Q$ remain identical across all initializations as expected, since it is estimated directly from fully observed state transitions.

\begin{table}[htbp]
\centering
\caption{Mean and standard deviation (in parentheses) of parameter estimates across random initializations}

\label{tab:robustness_params}
\begin{tabular}{l c l c l c}
\toprule
Parameter & Estimate & Parameter & Estimate & Parameter & Estimate \\
\midrule
$\hat P^{(0)}(0,0)$ & $0.797\ (0.009)$ &
$\hat P^{(1)}(0,0)$ & $0.721\ (0.008)$ &
$\hat P^{(2)}(0,0)$ & $0.509\ (0.019)$ \\

$\hat P^{(0)}(0,1)$ & $0.203\ (0.009)$ &
$\hat P^{(1)}(0,1)$ & $0.279\ (0.008)$ &
$\hat P^{(2)}(0,1)$ & $0.491\ (0.019)$ \\

$\hat Q(0,0)$ & $0.778\ (-)$ &

$\hat Q(0,1)$ & $0.222\ (-)$ &
$\hat Q(1,1)$ & $0.678\ (-)$ 
   \\

   $\hat B(0, 0)$ & $0.735\ (0.012)$ &
$\hat B (1,1)$ & $0.800\ (0.000)$ &
 \\
\bottomrule
\end{tabular}
\end{table}
We employ a point-based value iteration algorithm to compute the optimal maintenance policy. 
The cost parameters are specified as follows: the operating costs for component $U_2$ are 
$\bm c_{o,2} = (5, 40)$, while those for component $U_1$ are 
$\bm c_{o1} = (10, 20, 30)$. 
The setup cost is $c_s = 10$, and the replacement costs for components $U_1$ and $U_2$ are 
$c_{r,1} = 30$ and $c_{r,2} = 100$, respectively. 
The belief of failure probability $\pi(1)$ is discretized on a uniform grid with step size $0.002$. 
The discount factor is set to $0.95$, and the value iteration algorithm is terminated when the convergence tolerance reaches $10^{-3}$.

In this example, the belief state of component $U_2$ is 
$\bm \pi = (\pi(0), \pi(1))$, 
where $\pi(i)$ denotes the belief that component $U_2$ is in state $i$, $i \in S_2 = \{0, 1\}$.  
In the following, we will focus on $\pi(1)$, 
$i.e$., the belief of the failure probability of component $U_2$, 
 rather than the full tuple $\bm \pi$, to illustrate the characteristic of the maintenance policy with respect to the system state.

Figures~\ref{fig:MaintenanceO} and~\ref{fig:MaintenanceEs} present the optimal maintenance policies obtained using the true parameter values and the mean estimates, respectively, where 
$\pi(1)$ denotes the  failure probability belief of component $U_2$. The two policies exhibit highly consistent structural features, indicating that parameter estimation uncertainty has  a limited impact on the qualitative form of the optimal policy. In particular, the threshold-based structure is preserved across both cases. The corresponding value functions $V((1,0),0)$ under the true parameters and the estimated parameters are 
$950.59$
and 
$976.74,$  respectively,
representing a difference of only 
$2.75\%,$ indicating that the parameter estimates are reliable from a decision-making perspective.

\begin{figure}[!htp]
    \centering
    \includegraphics[width=0.9\linewidth]{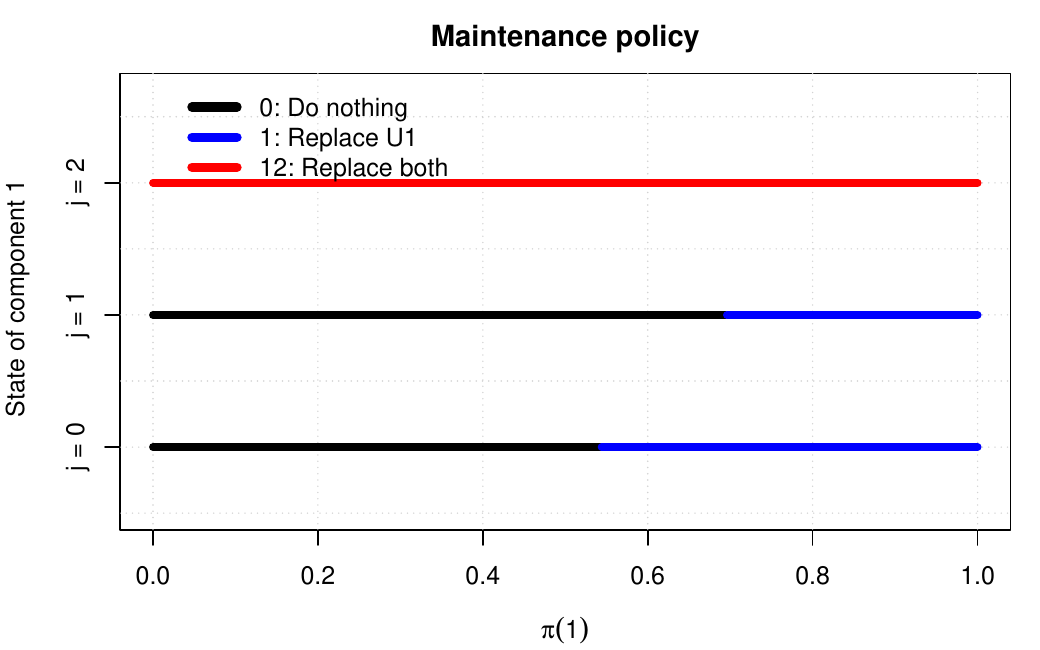}
    \caption{Maintenance policy using the true parameters}
    \label{fig:MaintenanceO}
\end{figure}

\begin{figure}[!htp]
    \centering
    \includegraphics[width=0.9\linewidth]{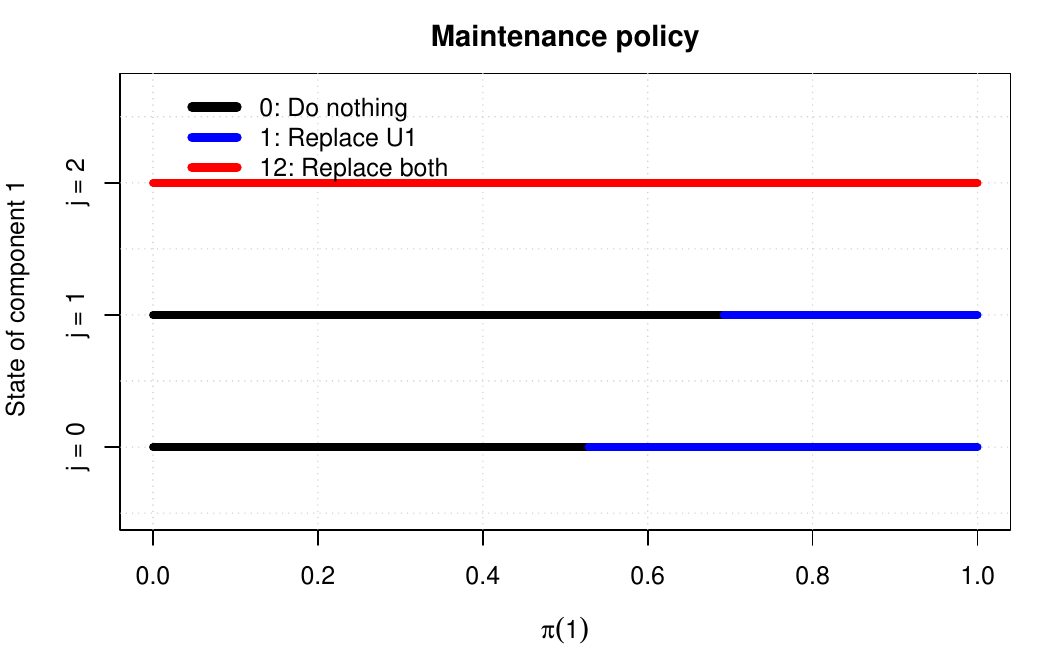}
    \caption{Maintenance policy using the mean parameter estimates}
    \label{fig:MaintenanceEs}
\end{figure}

\subsection{Policy comparison}\label{sec_policycompare}

In this subsection, we compare three preventive maintenance  and replacement policies for a two-component system consisting of components $U_1$
and $U_2$.
The policies differ in how replacement decisions are triggered, the extent to which system states are jointly considered, and the complexity of the resulting optimization problem. The objective of all policies is to minimize the long-run expected cost, represented by the value function $V(\bm 0, 0),$
 starting from the initial system state.

{\bf Policy 1} (Single-Threshold Policy) is a partially joint decision-making strategy in which preventive replacement of component $U_2$ is governed by a fixed failure probability threshold, while decisions regarding  component $U_1$ are optimized dynamically using a Markov Decision Process (MDP) framework. Under this policy, at each decision epoch, the system monitors the posterior failure probability of  component $U_2$, denoted by $\pi(L_2)$, where $L_2$ is its failure threshold. If this failure probability exceeds a predetermined threshold $\xi$, $U_2$ is immediately replaced, regardless of the condition of  component $U_1$. In contrast, the maintenance decision for component $U_1$ is made within a full MDP framework. At each decision epoch, the controller chooses whether to perform maintenance or do-nothing, based on its current state. 
The optimal preventive maintenance threshold for component $U_2$, denoted by $\xi^*$, is determined by minimizing the expected total cost starting from the initial state:
$\xi^*=\text{arg min}_{\xi}(V((1,0), 0))$.

{\bf Policy 2} (Double-threshold Policy) extends Policy 1 by introducing explicit replacement thresholds for both components, thereby increasing coordination between maintenance actions while retaining a relatively simple decision structure.
Under this policy, component $U_2$ is replaced whenever its posterior failure probability exceeds a predetermined threshold $\xi$, similar to Policy 1. In addition, component  $U_1$ is preventively replaced when its degradation level reaches or exceeds a predefined threshold
$\upsilon$. The joint replacement strategy allows for more structured decision-making and can capture some dependency effects between the two components, especially if joint replacement or shared setup costs are considered. For example, when both components approach their respective thresholds simultaneously, coordinated maintenance may reduce the setup cost. The optimal threshold $(\xi^*, \upsilon^*) $ is given by
  $(\xi^*, \upsilon^*)=\text{arg min}_{\xi,\upsilon}(V((1,0), 0)).$
  
{\bf Policy 3} (Independent Policy) represents a simplified benchmark approach in which components $U_1$
and $U_2$ are assumed to be statistically and operationally independent. Under this assumption, the degradation process of component $U_1$ provides no information about the condition or failure probability of component $U_2$.
As a result, maintenance decisions for each component are made independently using single-component optimization models.

 
  
The main performance metric  is the gap between long-run expected discounted cost between our policy, named Policy 0 in the following and the compared policies 1-3.
We take $$\%\textit{Gap}_{i}=100\cdot V((1,0), 0; \textit{Policy} ~i)-V((1,0), 0; \text{Policy} ~0)/V((1, 0), 0; \textit{Policy} ~0)),$$ 
as the policy evaluation indicator. For $i=0,1, 2$, $V((1,0), 0; \textit{Policy} ~i)$ is the expected discounted expected cost with initial state $((1,0), 0)$ under policy $i$.
 $V((1,0), 0; \textit{Policy} ~3)$ is the value function obtained by evaluating the policy developed under the assumption that $U_1$ and $U_2$ are independent. 
 The input parameters are given in Table \ref{tab:input_parameters}.

\begin{table}[htbp]
\centering
\caption{Input parameters}
\label{tab:input_parameters}

\renewcommand{\arraystretch}{1.15}   
\setlength{\tabcolsep}{9pt}          
\small                               

\begin{tabular}{>{\raggedright\arraybackslash}p{3.3cm} c p{9.2cm}}
\toprule
\textbf{Input parameter} & \textbf{Choices} & \textbf{Values} \\
\midrule

\multirow{2}{=}{Operation cost} 
 & 2 & $\bm c_{o1}=(1,1.5,2,2.5,3,3.5),\quad \bm c_{o2}=(1,10)$ \\
 &   & $\bm c_{o1}=(1,2,3,4,5,6),\hspace{1cm}\bm  c_{o2}=(1,8)$ \\

\midrule
Replacement cost of $U_1$ & 2 & $c_{r, 1}=15,\ 8$ \\

\midrule
Replacement cost of $U_2$ & 2 & $c_{r,2}=8,\ 6$ \\

\midrule
Observation matrix of $U_2$ & 2 & 
$\bm B=\begin{bmatrix}0.8 & 0.2\\0.2 & 0.8\end{bmatrix}, \quad
\bm B=\begin{bmatrix}0.7 & 0.3\\0.3 & 0.7\end{bmatrix}$ \\

\midrule
\multirow{2}{=}{Transition matrix of $U_2$}
 & 2 & 
$\bm P^{(i)}=\begin{bmatrix}0.9 & 0.1\\0 & 1\end{bmatrix},\ i=0,1,2; \quad
\bm P^{(i)}=\begin{bmatrix}0.8 & 0.2\\0 & 1\end{bmatrix},\ i=3,4,5.$ \\

 & 2 & 
$\bm P^{(i)}=\begin{bmatrix}0.7 & 0.3\\0 & 1\end{bmatrix},\ i=0,1,2; \quad
\bm P^{(i)}=\begin{bmatrix}0.6 & 0.4\\0 & 1\end{bmatrix},\ i=3,4,5.$ \\

\bottomrule
\end{tabular}
\end{table}

In choice 1,  the transition matrices of $U_1$ is as follows.
\[
\bm Q =
\begin{bmatrix}
0.8 & 0.04 & 0.04 & 0.04 & 0.04 & 0.04 \\
0.0 & 0.8  & 0.05 & 0.05 & 0.05 & 0.05 \\
0.0 & 0.0  & 0.8  & \tfrac{1}{15} & \tfrac{1}{15} & \tfrac{1}{15} \\
0.0 & 0.0  & 0.0  & 0.8  & 0.1  & 0.1 \\
0.0 & 0.0  & 0.0  & 0.0  & 0.9  & 0.1 \\
0.0 & 0.0  & 0.0  & 0.0  & 0.0  & 1.0
\end{bmatrix}
\]
In choice 2, $\bm Q$ is as follows.
\[
\bm Q =
\begin{bmatrix}
0.6 & 0.08 & 0.08 & 0.08 & 0.08 & 0.08 \\
0.0 & 0.6  & 0.1  & 0.1  & 0.1  & 0.1 \\
0.0 & 0.0  & 0.6  & \tfrac{2}{15} & \tfrac{2}{15} & \tfrac{2}{15} \\
0.0 & 0.0  & 0.0  & 0.6  & 0.2  & 0.2 \\
0.0 & 0.0  & 0.0  & 0.0  & 0.6  & 0.4 \\
0.0 & 0.0  & 0.0  & 0.0  & 0.0  & 1.0
\end{bmatrix}
\]

\begin{table}[htbp]
\centering
\caption{$GAP\%$ comparison}
\label{tab:gap_comparison}
\renewcommand{\arraystretch}{1.2} %
\setlength{\tabcolsep}{2pt}     %
\footnotesize                    %
\resizebox{1\textwidth}{!}
{    
\begin{tabular}{>{\raggedright\arraybackslash}p{3.2cm} c ccc ccc ccc}
\toprule
\textbf{Input} & \textbf{Choice} &
\multicolumn{3}{c}{\textit{Policy} 1} &
\multicolumn{3}{c}{\textit{Policy} 2} &
\multicolumn{3}{c}{\textit{Policy} 3} \\
\cmidrule(lr){3-5}\cmidrule(lr){6-8}\cmidrule(lr){9-11}
 &  & Min & Mean & Max & Min & Mean & Max & Min & Mean & Max \\
\midrule
\multirow{2}{*}{\makecell{Operating~~cost}}
 & 1 & 0.44 & 1.93 & 3.61 & 1.80 & 3.68 & 5.72 & 0.10 & 0.83 & 2.39 \\
 & 2 & 0.47 & 1.81 & 3.65 & 1.28 & 2.89 & 5.54 & 0.10 & 0.35 & 1.10 \\
\midrule
\multirow{2}{*}{\makecell{Replacement~cost $U_1$}}
 & 1 & 0.56 & 1.99 & 3.65 & 1.28 & 3.73 & 5.72 & 0.09 & 0.43 & 1.22 \\
 & 2 & 0.44 & 1.75 & 3.61 & 1.35 & 2.84 & 3.98 & 0.10 & 0.76 & 2.39 \\
\midrule
\multirow{2}{*}{\makecell{Replacement~cost $U_2$}}
& 1 & 0.47 & 2.07 & 3.65 & 1.80 & 3.69 & 5.72 & 0.09 & 0.60 & 2.30 \\
& 2 & 0.44 & 1.68 & 3.41 & 1.23 & 2.88 & 4.82 & 0.10 & 0.59 & 2.39 \\
\midrule
\multirow{2}{*}{\makecell{Observation~matrix of $U_2$}}
 & 1 & 0.44 & 1.60 & 3.55 & 1.44 & 3.01 & 5.54 & 0.09 & 0.58 & 2.39 \\
 & 2 & 0.86 & 2.14 & 3.65 & 1.28 & 3.56 & 5.72 & 0.13 & 0.61 & 2.20 \\
\midrule
\multirow{2}{*}{\makecell{Transition~matrix of $U_1$}}
 & 1 & 0.44 & 1.51 & 3.26 & 0.09 & 0.77 & 2.39 & 1.28 & 2.94 & 4.92 \\
 & 2 & 0.90 & 2.23 & 3.65 & 0.09 & 0.42 & 1.07 & 1.73 & 3.64 & 5.72 \\
\midrule
\multirow{2}{*}{\makecell{Transition~matrix of $U_2$}}
& 1 & 0.44 & 1.25 & 2.73 & 1.44 & 3.14 & 5.55 & 0.10 & 0.55 & 2.39 \\
& 2 & 1.25 & 2.50 & 3.65 & 1.28 & 3.43 & 5.72 & 0.01 & 0.63 & 1.50 \\
\midrule
\textbf{Total} &  & 0.44 & 1.87 & 3.65 & 1.28 & 3.29 & 5.72 & 0.09 & 2.39 & 5.93 \\
\bottomrule
\end{tabular}
}
\end{table}

Table~\ref{tab:gap_comparison} presents the percentage cost gap (\(GAP\%\)) between the proposed \textbf{Policy~0} and \textbf{Policy~1--3}. 
The values of “Min”, “Mean”, and “Max” correspond to the minimum, average, and maximum percentage increase in the long-run expected discounted cost of each benchmark relative to our policy, under the given range of system parameters.

Across all test scenarios, the proposed Policy~0 consistently achieves the lowest expected cost. 
The competing strategies yield cost increases in the range of approximately 
$0.4\text{--}3.6\%$ (Policy~1), $1.3\text{--}5.7\%$ (Policy~2), and $0.1\text{--}5.9\%$ (Policy~3), respectively. 
The total mean $GAP\%$ of 1.87\%, 3.29\%, and 2.39\% for the three  policies indicates that our method outperforms all alternatives across different parameter combinations. 
This demonstrates the efficiency of the POMDP-based maintenance policy, which optimally couples the hidden component $U_2$ with the observable component $U_1$ through their transition dependence.

A consistent trend can be observed between  Policy~1 and Policy~2. 
{Policy~1 (single-threshold)} always performs better than {Policy~2 (double-threshold)} in every tested configuration. 
For example, in the “Operating cost” case, the mean cost increase under Policy~1 is 1.93\%, compared with 3.68\% under Policy~2. 
Similarly, for "Replacement cost $U_2$,'' the mean values are 2.07\% versus 3.69\%, respectively. 
This dominance arises because the double-threshold strategy imposes stricter and more restrictive decision rules—each component’s preventive replacement threshold is predetermined—limiting flexibility in maintenance actions. 
Therefore, the single-threshold approach (Policy~1) achieves a better balance between responsiveness and cost efficiency.
{Policy~3 (independent policy)} assumes statistical independence between $U_1$ and $U_2$. 
In many cases, this simplification still yields competitive results: for example, the mean gaps for “Operating cost” (0.83\%) and “Observation matrix of $U_2$”  (0.58\%) are notably smaller than those of policy 1-2. 
However, its performance is unstable and highly sensitive to the transition dynamics of $U_1$. 
While the independent policy is easy to implement and computationally efficient, it ignores potential correlations or dependencies between the components. Consequently, it may lead to suboptimal system-level performance, particularly in settings where component interactions significantly affect degradation or failure behavior. When the degradation of $U_1$ strongly influences the hidden deterioration rate of $U_2$ (as in the “Transition matrix of $U_1$” case), the independence assumption becomes invalid. 
In this case, the mean and maximum gaps rise sharply to 2.94\% and 5.72\%, respectively, which is significantly higher than in other configurations, demonstrating that ignoring stochastic dependence leads to poor decisions.

\subsection{Sensitivity analysis}
Next, some sensitivity analyses are presented based on the system parameters in section \ref{sec_policycompare}.
We take the first choice in Table \ref{tab:input_parameters} as the base case, where the first choice of parameters are chosen for all parameters.
Figure~\ref{fig:base} presents the optimal maintenance policy for the base case. The resulting policy structure is consistent with Theorem \ref{thm_3}, as the boundary
$\bm \pi^*_{0,12}(j)$ is non-increasing in the degradation state of component $U_1$. In particular, replacement of both components becomes optimal at a lower posterior failure probability of component $U_2$ when $j=4$ than $j=3,$ , indicating that more severe degradation of component $U_1$ accelerates the incentive for joint replacement.
\begin{figure}
    \centering
\includegraphics[width=0.9\linewidth]{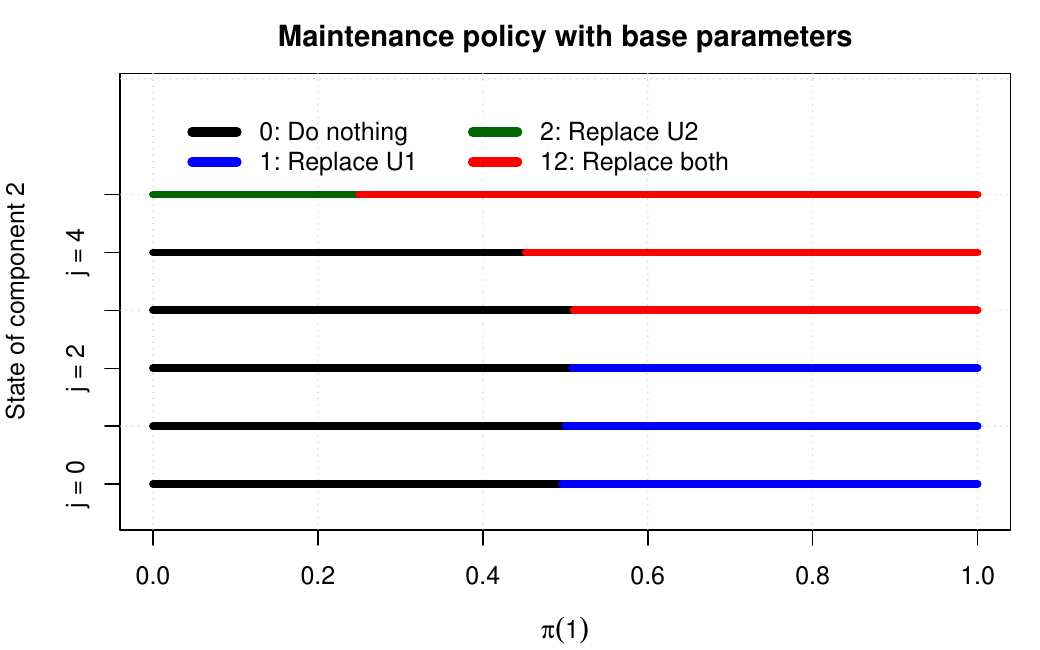}
    \caption{The optimal maintenance policy with the base parameters}
    \label{fig:base}
\end{figure}

\begin{figure}[!hbt]
    \centering
    \includegraphics[width=0.9\linewidth]{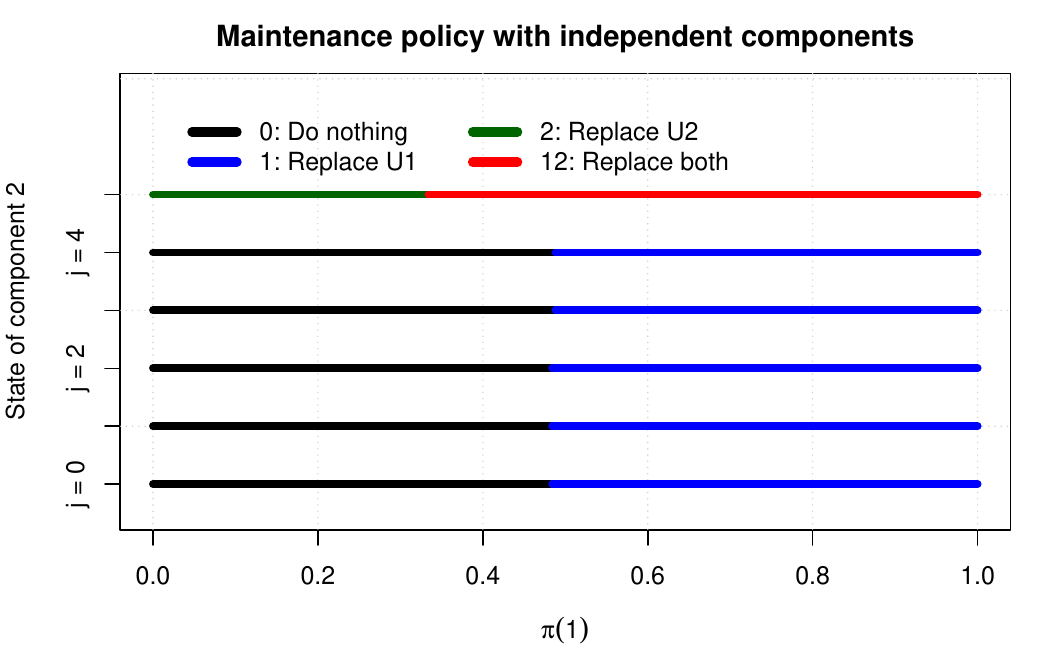}
    \caption{The optimal maintenance policy with independent components}
\label{fig:indepedent}
\end{figure}
Figure~\ref{fig:indepedent} presents the optimal maintenance policy for the case in which components $U_1$ and $U_2$ are independent. In this setting, all transition matrices $\bm P^{(j)}$
are set equal to $\bm P^{(0)}$, resulting in a system that is more reliable than the base case. Because component $U_2$ exhibits slower degradation dynamics, the system tends to postpone joint replacement and instead favors replacing only component $U_1$ until the posterior failure probability of component $U_2$ becomes relatively large, compared with the base system.

For intermediate degradation states of component~$U_1$, it is no longer optimal to replace both components even when the posterior failure probability of component~$U_2$ is high; rather, replacing only component~$U_1$ becomes preferable. The optimal value function at the initial state, $V((1,0), 0)$
 is 128.11, which is lower than the corresponding value in the base case,
137.03.

\begin{figure}[!htp]
    \centering
    \includegraphics[width=0.9\linewidth]{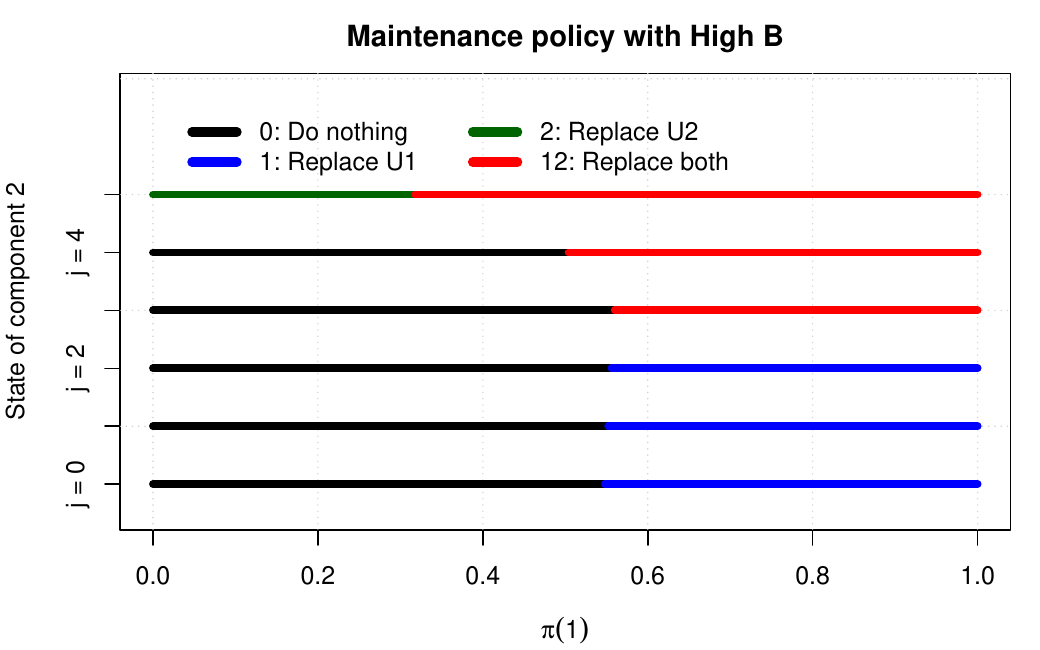}
    \caption{The optimal maintenance policy with more accurate observation}
    \label{fig:betterB}
\end{figure}

\begin{figure}[!htp]
    \centering
\includegraphics[width=0.9\linewidth]{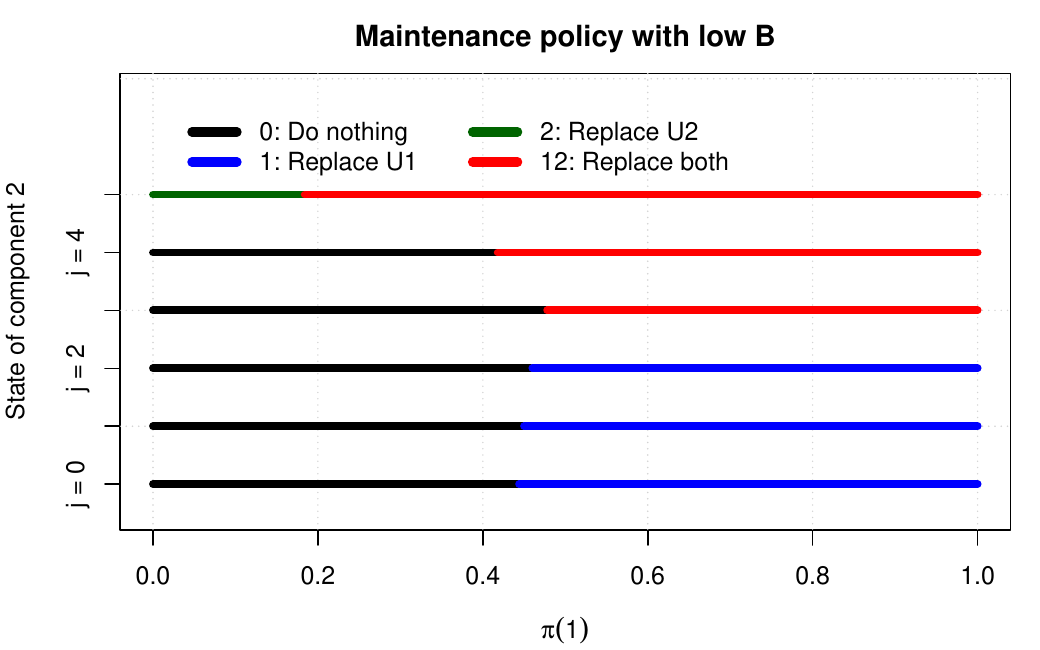}
    \caption{The optimal maintenance policy with less accurate observation}
    \label{fig:lowB}
\end{figure}
Figures~\ref{fig:betterB} and~\ref{fig:lowB} illustrate how the optimal maintenance policy varies with the observation precision of component~$U_2$. Compared with the base case, the overall policy structure remains largely unchanged, indicating structural robustness with respect to observation quality. In addition, higher observation precision leads to larger maintenance thresholds in terms of the posterior failure probability of component~$U_2$ for each degradation state of component $U_1$. This occurs because more accurate observations reduce uncertainty about the system’s future evolution, thereby allowing maintenance actions to be postponed until higher levels of failure risk are reached.

\section{Conclusions}\label{sec_conclusion}
In this work, we have investigated the maintenance optimization of a two-component dependent system in a POMDP framework. The system exhibits unidirectional positive degradation dependence, in which the health
state of component $U_1$ influences the degradation process of component $U_2$, but not
vice versa. We have analytically shown the properties of the value function and the optimal maintenance policy. A Baum-Welch algorithm with multiple sample paths is given for the parameter estimates. Through numerical experiments, we demonstrate the robustness of the parameter estimation procedure and systematically compare the performance of the proposed maintenance policy with several benchmark strategies. 

Further research may focus on the following issues. First, this work considers small scale system size, where the point based value iteration algorithm can be used. POMDPs generally face the curse of dimensionality, it is interesting to propose learning methods to solve large-scale problems \cite{drent2024optimal}.

Secondly, this work considers a system with two components, for a more general multi-component system, its maintenance policy in POMDP framework is challenging and interesting \cite{vora2023welfare}. 
In addition, constraints on maintenance accessibility, maintenance delay, and maintenance material availability can be considered in the future.

\bibliographystyle{apalike}
\bibliography{NewModel}



\section*{Acknowledgements}
This work is supported by the National Natural Science Foundation of China (Grant No.72371027), the Spanish Ministry of Science and Innovation - State Research Agency, under grants  PID2024-156234NB-C21 and PID2024-156234NB-C22; and IMAG-Maria de Maeztu grant CEX2020-001105-/AEI/10.13039/501100011033.
\end{document}